\documentclass[11pt, reqno]{amsart}
\usepackage{class_fields}

\title{On abelian covers of the projective line with fixed gonality and many rational points}
\author{Xander Faber}
\address{Institute for Defense Analyses \\
Center for Computing Sciences \\
17100 Science Drive \\
Bowie, MD}
\email{awfaber@super.org}
\author{Floris Vermeulen}
\address{KU Leuven \\
  Department of Mathematics \\
  Celestijnenlaan 200B \\
  Leuven 3001, Belgium }
\email{floris.vermeulen@kuleuven.be}

\begin{document}
\begin{abstract}
  A smooth geometrically connected curve over the finite field $\FF_q$ with
  gonality $\gamma$ has at most ${\gamma(q+1)}$ rational points. The
  first author and Grantham conjectured that there exist curves of every
  sufficiently large genus with gonality $\gamma$ that achieve this bound. In
  this paper, we show that this bound can be achieved for an infinite sequence
  of genera using abelian covers of the projective line. We also argue that
  abelian covers will not suffice to prove the full conjecture.
\end{abstract}
\maketitle


\section{Introduction}
Throughout, we use the unqualified term ``curve'' to mean a smooth proper
geometrically connected scheme of dimension~$1$ over a field. Given a curve $C$
over a finite field $\FF_q$, the existence of a morphism $f \colon C \to \PP^1$
provides a bound on the number of points of $C$. Indeed, every rational point of
$C$ maps to one of the rational points of $\PP^1$, and each fiber contains at
most $\deg(f)$ points:
\begin{equation}
  \label{eq:morphism_bound}
  \#C(\FF_q) \le \deg(f) \; (q+1).
\end{equation}
The minimum degree of a morphism to $\PP^1$ is known as the \textbf{gonality}
of~$C$.

Write $N_q(g,\gamma)$ for the maximum number of points on a curve $C_{/\FF_q}$
with genus $g$ and gonality $\gamma$; we set this quantity to $-\infty$ if no
such curve exists. From~\eqref{eq:morphism_bound}, we know that $N_q(g,\gamma)
\le \gamma (q+1)$. In \cite{Faber_Grantham_GF2,Faber_Grantham_GF34}, the first
author and Grantham gave many examples of curves where this bound is achieved,
which led to the following conjecture:

\begin{conjecture}
  \label{conj:large_genus}
  Fix a prime power $q$ and an integer $\gamma \ge 2$. For $g$ sufficiently
  large, we have
  \[
  N_q(g,\gamma) = \gamma(q+1).
  \]
\end{conjecture}

The implicit challenge in the conjecture is to construct interesting families of
curves. The first author and Grantham wrote down explicit hyperelliptic
equations to handle the case $\gamma = 2$ \cite{Faber_Grantham_GF34}.  The
second author established the conjecture for $q$ odd by constructing singular
curves on toric surfaces \cite{vermeulen2022curves}. The proof uses a
squarefree-discriminant trick to insure that the singular curves are smooth
above non-rational points of $\PP^1$; in characteristic~$2$, discriminants are
never squarefree, so one must find a different approach to control singularities.
The goal of this note is to investigate how much can be proved using only
abelian covers of the projective line, especially in the remaining open case
where $q$ is even. Recall that an \textbf{abelian cover} of curves $X \to Y$ is
a nonconstant morphism for which the corresponding extension of function fields
$\kappa(X) / \kappa(Y)$ is Galois with abelian group of automorphisms. See
\cite[\S5.12]{Serre_Rational_Points_Book} and \cite[\S6.3]{TVN_advanced} for
summaries of geometric class field theory; see \cite[Ch.~V]{milneCFT} and
\cite{Serre_Algebraic_Groups} for more robust treatments.

\begin{theorem}
  \label{thm:main}
  Fix a prime power $q$ and an integer $\gamma \ge 2$. There exists an infinite
  sequence of distinct positive integers $(g_i)$ and abelian covers $C_i \to
  \PP^1$ of degree $\gamma$ such that $C_i$ has genus $g_i$ and $\gamma(q+1)$
  rational points. In particular,
  \[
     \limsup_{g \to \infty} N_q(g,\gamma) = \gamma(q+1).
  \]
\end{theorem}

Unfortunately, abelian covers are not sufficient to replace the limsup with a
proper limit. For example, consider the case of covers of $\PP^1$ of prime
degree $\gamma \equiv 1 \pmod{4}$, where $\gamma$ does not divide $q$. If $C \to
\PP^1$ is abelian of degree~$\gamma$, then all of the ramification indices are
equal to~$1$ or~$\gamma$. Let $m$ be the number of ramified points of $C$. The
Hurwitz formula shows the genus of~$C$ is $\frac{\gamma-1}{2}(m-2)$. This
expression is always even, so abelian covers of degree~$\gamma$ will never have
anything to say about curves of odd genus.

Roughly speaking, there are three main approaches in the literature to
constructing curves of large genus with many rational points:
\begin{itemize}
  \item Modular curves, especially in towers;
  \item Class field theory constructions; and
  \item \textit{Ad hoc} methods for special families. 
\end{itemize}
Modular methods typically produce curves whose genus and gonality grow together,
which is not helpful for Conjecture~\ref{conj:large_genus}. As indicated in the
preceding paragraph, methods of class field theory do not produce enough curves
to address the full scope of the conjecture. This leaves \textit{ad hoc}
constructions. One particularly powerful construction is that of curves on toric
surfaces; this was first introduced for the study of rational points in
\cite{KWZ} and later used to prove Conjecture~\ref{conj:large_genus} for odd $q$
in \cite{vermeulen2022curves}. See \cite[p.vii]{Serre_Rational_Points_Book} and
\cite[p.1]{TVN_advanced} for more philosophical exposition on constructing
curves over finite fields.

\medskip

\noindent
\textbf{Acknowledgments.} We thank Congling Qiu for spotting an omission in the
hypotheses of Lemma~\ref{lem:branch} in the version of this paper that appeared
in the \textit{International Journal of Number Theory} in 2022.

  

\section{Proof of Theorem~\ref{thm:main}}

In our discussion below, we will identify $\PP^1$ with $\Spec \FF_q[t] \cup
\{\infty\}$. Recall that the \textbf{branch locus} of a finite morphism $f
\colon X \to Y$ is the minimal closed subset $B \subset Y$ such that the induced
morphism $X \smallsetminus f^{-1}(B) \to Y \smallsetminus B$ is
\'etale. Equivalently, the branch locus is the support of the discriminant of
the induced extension of function fields $\kappa(X) / \kappa(Y)$.

The construction of our sequence of curves involves three steps:
\begin{itemize}
  \item Choose a morphism $h \colon \PP^1 \to \PP^1$ that maps every point of
    $\PP^1(\FF_q)$ to $\infty$.
  \item Construct a sequence of morphisms $f_i \colon X_i \to \PP^1$ of degree
    $\gamma$ such that $\infty$ splits completely in each $X_i$ and the genera
    of the $X_i$ tend to infinity.
  \item Take the fiber product of $f_i$ and $h$ to obtain a new covering of
    curves $C_i := X_i \times_{\PP^1} \PP^1 \to \PP^1$ of degree~$\gamma$ with the
    correct number of rational points. 
\end{itemize}
We must take care in choosing the branch loci of the morphisms $f_i$ in order to
have control over the genus of the curve $C_i$.

The construction of $h$ is explicit. Let $a \in \FF_q[t]$ be an irreducible
polynomial of degree $q+1$ and define
\[
   h(t) = \frac{a(t)}{t^q - t}.
\]
Then $h$ has degree $q+1$, and each point of $\PP^1(\FF_q)$ maps to
$\infty$. 

\begin{lemma}
  \label{lem:abelian}
  Fix $\gamma \ge 1$ and a proper closed subset $S \subset \PP^1$.  There exists
  an abelian cover $f \colon X \to \PP^1$ of degree $\gamma$ such that the point
  $\infty$ splits completely in $X$, and such that the branch locus of $f$ is
  supported on a closed point outside $S$.
\end{lemma}

\begin{proof}
  Define a modulus $\fm = e[P]$, where $P$ is a yet-to-be-determined closed
  point of $\PP^1$ and $e \ge 1$ is an integer. By \cite[Thm.~V.1.7]{milneCFT},
  the order of the ray class group $\Cl_{\fm}^0 (\PP^1)$ is given by
  \[
    \#\Cl_{\fm}^0(\PP^1) = \frac{q^{\deg(P)} - 1}{q-1}
    q^{\deg(P)(e-1)}.
  \]
  This quantity is divisible by $\gamma$ for infinitely many closed points $P$
  and choices of $e \ge 1$. In particular, we may choose $P$ and $e$ so that
  $\gamma$ divides the order of $\#\Cl_{\fm}^0(\PP^1)$ and so that $P \not\in
  S \cup \{\infty\}$.

  Let $G$ be a quotient of $\Cl_{\fm}^0(\PP^1)$ of order $\gamma$.  Then the map
  $D \mapsto D - \deg(D) [\infty]$ followed by projection to $G$ gives a
  surjective homomorphism ${\alpha \colon \Cl_{\fm}(\PP^1) \to G}$. We now
  invoke the existence theorem of geometric class field theory
  \cite[p.124]{Serre_Rational_Points_Book} to obtain a curve $X_{/\FF_q}$ and a
  Galois covering  $X \to \PP^1$ with group $G$ whose branch locus is
  supported on $P$, and whose Frobenius element for $Q \ne P$ is given by
  $\alpha([Q])$. By construction, we have $\alpha([\infty]) = 1$, so that
  $\infty$ splits completely in~$X$.
\end{proof}

For a curve $C$, write $g(C)$ for its genus. 

\begin{lemma}
  \label{lem:branch}
  Let $f \colon Y \to X$ and $h \colon Z \to X$ be morphisms of curves over a
  field $k$ whose branch loci in $X$ are disjoint. Then the fiber product $Y
  \times_X Z$ is a smooth $k$-variety of dimension 1. If it is geometrically
  irreducible, then it has genus
   \[
     d_h g(Y) + d_f g(Z) - d_f d_h g(X) + (d_f - 1)(d_h - 1),
     \]
   where $d_f = \deg(f)$ and $d_h = \deg(h)$.   
\end{lemma}

\begin{proof}
  Write $W = Y \times_X Z$ for ease of notation. The following diagram describes
  our situation:
  \[
  \begin{tikzcd}
    W \ar[r,"\hat f"] \ar[d,"\hat h"] & Z \ar[d,"h"] \\
    Y \ar[r,"f"] & X
  \end{tikzcd}
  \]
  Let $B_f$ and $B_h$ be the branch loci of $f$ and $h$, respectively.  The
  base extension of a smooth morphism is smooth, so $\hat f$ is smooth above $Z
  \smallsetminus h^{-1}(B_f)$. Similarly, $\hat h$ is smooth above $Y
  \smallsetminus f^{-1}(B_h)$. Since the branch loci of $f$ and $h$ are
  disjoint, we conclude that $W$ is regular.

  For the genus formula, write $R_Y$, $R_Z$, and $R_W$ for the ramification
  divisors of the morphisms $f$, $h$, and $f \hat h = h \hat f$,
  respectively. Since $f$ and $h$ have  disjoint branch loci, we learn that
  \[
     R_W = \hat f^* R_Z + \hat h^* R_Y.
  \]
  Taking degrees gives
  \[
     \deg(R_W) = \deg(f) \deg(R_Z) + \deg(h) \deg(R_Y).
     \]
  The Hurwitz formulae for the morphisms $f \colon Y \to X$, $h \colon Z
  \to X$, and $f \hat h \colon W \to X$ show that
  \begin{align*}
    \deg(R_Y) &= 2g(Y) - 2 - \deg(f) \left(2g(X) - 2\right) \\
    \deg(R_Z) &= 2g(Z) - 2 - \deg(h) \left(2g(X) - 2\right) \\
    \deg(R_W) &= 2g(W) - 2 - \deg(f) \deg(h) \left(2g(X) - 2\right).    
  \end{align*}
  Solving for $g(W)$ in this system of four equations gives the desired result.
\end{proof}

Now we complete the proof of the Theorem~\ref{thm:main}. Let $B_0$ be the branch
locus of $h \colon \PP^1 \to \PP^1$. Apply Lemma~\ref{lem:abelian} to construct
an abelian cover $f_0 \colon X_0 \to \PP^1$ of degree $\gamma$ such that
$\infty$ splits completely in $X_0$, and such that the branch locus of $f_0$ is
disjoint from $B_0$. Next we iteratively apply Lemma~\ref{lem:abelian} to
construct a sequence of abelian covers $f_i \colon X_i \to \PP^1$ of degree
$\gamma$ and an increasing sequence of closed subsets $B_0 \subsetneq B_1
\subsetneq B_2 \subsetneq \cdots \subsetneq \PP^1$ such that
\begin{itemize}
  \item $\infty$ splits completely in $X_i$;
  \item the branch locus of $f_i$ is a closed point disjoint from $B_i$; and
  \item the set $B_{i+1}$ is the union of $B_i$ and the branch locus of $f_i$.
\end{itemize}
Since the branch locus of $f_i$ is disjoint from the branch locus of $h$, we may
apply Lemma~\ref{lem:branch} to see that $C_i = X_i \times_{\PP^1} \PP^1$ is a
smooth 1-dimensional variety over $\FF_q$.

We argue next that $C_i$ is geometrically connected. Note that $C_i$ is a closed
subscheme of $X_i \times \PP^1$. The intersection theory of ruled surfaces
\cite[V.2.3]{Hartshorne_Bible} shows that the group of numerical equivalence
classes of $X_i \times \PP^1$ is generated by $e_1 = X_i \times \{\text{pt}\}$
and $e_2 = \{pt\} \times \PP^1$, and these two curves satisfy $e_1^2 = e_2^2 =
0$ and ${e_1.e_2 = 1}$. Since $C_i$ is a fiber product, each of its geometric
components dominates $X_i$ and $\PP^1$ for the respective projection
maps. Consequently, any geometric component is numerically equivalent to $ae_1 +
be_2$ for some $a,b > 0$. If $C_i$ has two distinct geometric components, then
those components have nonzero intersection. As any intersection point between
components is non-smooth, we have a contradiction. It follows that $C_i$ is
geometrically connected.

Now we argue that the genus of $C_i$ tends to infinity with $i$.
Lemma~\ref{lem:branch} shows that $C_i$ has genus
\[
    g(C_i) = (q+1) g(X_i) + q(\gamma-1).
\]
The branch locus of $f_i \colon X_i \to \PP^1$ is supported on a closed point
that is disjoint from the branch loci of all $f_j$ with $j < i$. As there are
only finitely many closed points of a given degree, the degree of this branch
locus must tend to infinity with $i$. By the Hurwitz formula, $g(X_i) \to
\infty$ with $i$, and hence the genus of $C_i$ must also tend to infinity with
$i$.

Finally, we show that $C_i$ has $\gamma(q+1)$ rational points and gonality
$\gamma$. Consider the following fiber product
diagram:
  \[
  \begin{tikzcd}
    C_i \ar[r,"\hat f_i"] \ar[d,"\hat h"] & \PP^1 \ar[d,"h"] \\
    X_i \ar[r,"f_i"] & \PP^1
  \end{tikzcd}
  \]
Note first that $\hat f_i$ has the same degree as $f_i$, namely $\gamma$; thus
$C_i$ has gonality at most $\gamma$. Since $\infty$ splits completely for the
map $f_i \colon X_i \to \PP^1$, every pre-image of $\infty$ under $h$ splits
completely in $C_i$ under the map $\hat f_i$
\cite[Prop.~3.9.6b]{Stichtenoth_book}. In particular, since $h$ maps all
rational points of $\PP^1$ to $\infty$, we see that $\hat f_i^{-1}(x)$ consists
of $\gamma$ rational points for each $x \in \PP^1(\FF_q)$. Thus $\#C_i(\FF_q)
\ge \gamma(q+1)$. By \eqref{eq:morphism_bound}, $C_i$ must have gonality at
least $\gamma$. We conclude that $C_i$ has gonality $\gamma$ and exactly $\gamma
(q+1)$ rational points, and the proof of Theorem~\ref{thm:main} is complete.

\bibliographystyle{plain}
\bibliography{class_fields}

\end{document}